\documentclass[12pt]{article}
\usepackage{mathrsfs}
\usepackage{mathtools}
\usepackage{amssymb}
\usepackage{amsmath}
\usepackage{amsthm}
\usepackage{amscd}
\usepackage{commath}
\usepackage{enumerate}
\usepackage{stmaryrd}
\usepackage{enumitem}
\usepackage{hyperref}
\usepackage{thmtools}
\usepackage{listings}
\usepackage{tikz}
\usepackage{tikz-cd}
\usepackage{comment}
\usepackage{lipsum}
\usetikzlibrary{matrix,arrows,decorations.pathmorphing}
\usepackage{bm}
\theoremstyle{definition} 
\newtheorem{definition}{Definition}

\theoremstyle{plain} 
\newtheorem{theorem}{Theorem}
\newtheorem*{theorem*}{Theorem}
\newtheorem{conjecture}[theorem]{Conjecture}
\newtheorem{corollary}[theorem]{Corollary}
\newtheorem*{corollary*}{Corollary}
\newtheorem{lemma}[theorem]{Lemma}

\theoremstyle{remark} 
\newtheorem{remark}{Remark}
\newtheorem{notation}{Notation}

\newtheorem{question}{Question}
\newtheorem*{question*}{Question}

\newcommand{\I}{\mathscr I}

\newcommand{\wpi}{\widehat{\pi}}
\newcommand{\p}{p}

\DeclareMathOperator{\Is}{IS}

\title{{ Unlikely and just likely intersections for high dimensional families of elliptic curves}} 

\author{Asvin G.\footnote{\textit{Department of Mathematics; University of Wisconsin, Madison.}}
\footnote{\textit{email}: gasvinseeker94@gmail.com, \textit{website}: \href{https://asving.com/}{asving.com} }} 

\date{} 

\begin{document}





\maketitle 


\begin{abstract}
Given two varieties $V,W$ in the n-fold product of modular curves (over a finite field, number field or ring of integers of a number field), we answer affirmatively a question (formulated by Shou-Wu Zhang's AIM group) on whether the set of points in $V$ that are Hecke translations of some point on $W$ is dense in $V$. We need to make some (necessary) assumptions on the dimensions of the co-ordinate projections of $V,W$ but for instance, when $V$ is a divisor and $W$ is a curve, no further assumptions are needed.

We also examine the necessity of our assumptions in the case of unlikely intersections and show that, contrary to exceptions, two curves in a high dimensional space over a finite field can intersect infinitely often up to Hecke translations.
    
\end{abstract}






\section{Introduction}

In \cite{chai2012abelian}, Chai and Oort ask the following natural question: \emph{For every algebraically closed field $k$ and $g \geq 4$, does there exist an Abelian variety over $k$ of dimension $g$ that is not isogenous to the Jacobian of a stable curve?}. This conjecture was settled in characteristic $0$ by Tsimerman \cite{tsimerman2012existence} relying on the Andre-Oort conjecture about CM points in subvarieties. The analogous question over finite fields seems much harder since every point is a CM point. Nevertheless, the conjecture is believed to be true even over finite fields.

The above question can be interpreted as an \emph{unlikely-intersections} question for the Torelli locus inside $\mathcal A_g$, the moduli space of principally polarized abelian varieties in dimension $g$. In this form, it becomes natural to replace the Torelli locus with an arbitrary proper subvariety and retain the conjecture. In \cite{shankar2018unlikely}, Shankar and Tsimerman formulate a heuristic which predicts a negative answer for $g= 4$. We explain this heuristic in the context of $Y(1)^g$, the moduli space of dimension $g$ abelian varieties isomorphic to the product of $g$ elliptic curves instead of $\mathcal A_g$:

Let $\Is_n(m)$ be the set of isogeny classes corresponding to the abelian varieties parameterized by $Y(1)^n(\mathbb F_{q^m})$. \cite[Theorem 1.1]{dipippo1998real} shows that the size of the set $\Is_n(m)$ is about $q^{nm/2}$ (as $q^m \to \infty$) . It seems reasonable to treat the map $i: Y(1)^n(\mathbb F_{q^m}) \to \Is_n(m)$ sending a point $x \in Y(1)^n(\mathbb F_q)$ to the isogeny class of the corresponding abelian variety $A_x$ as a \emph{random} map and suppose that every isogeny class is of size $q^{nm/2}$ (ignoring problems related to supersingularity and the like). 

Let $V,W \subset Y(1)^n$ now be curves. For $n \geq 2$ and $m$ large enough, one would expect $|i(V(\mathbb F_{q^m}))| \sim q^m$ and $|i(V(\mathbb F_{q^m})) \cap i(W(\mathbb F_{q^m}))| \sim q^{2m}/q^{nm/2}$. Summing over all $m$, one expects that the number of points on $V$ geometrically isogenous to a point on $W$ is approximately $(1-q^2/q^{n/2})^{-1}$ which is a finite quantity (for $n \geq 5$).

As our first main result, we produce families of counter examples to this heuristic.
\theoremstyle{plain}
\newtheorem*{thm: ex1}{Theorem \ref{lem: postive low dim example}}

\begin{thm: ex1}
Let $V = (t^{a_1},\dots,t^{a_n})$ and $W = (t^{b_1},\dots,t^{b_n})$ for $a_i,b_i \geq 2$ (and without loss of generality, coprime to $p$) be two rational curves in $Y(1)_S^n$ for $S = \operatorname{Spec} \mathbb F_q$.

If the numbers $b_i/a_i$ are all powers of a fixed rational number $\gamma$, then there are infinitely many pairs of points on $V,W$ isogenous to each other.

Moreover, if we assume Artin's conjecture on primitive roots, then the same is true for arbitrary $a_i,b_i$.
\end{thm: ex1}

\theoremstyle{plain}
\newtheorem*{thm: ex2}{Theorem \ref{thm: ex2}}

\begin{thm: ex2}
Let $C = (t+b_0,t+b_1\dots,t+b_n), D = \Delta \subset Y(1)_S^{n+1}$ for $S = \operatorname{Spec} \mathbb F_q$ such that the $b_i\in \mathbb F_q$ are pairwise distinct and satisfy
\[\frac{b_i-b_0}{b_j-b_0} \in \mathbb F_p\]
for all $1 \leq i,j \leq n$. Then there are infinitely many points on $C$ isogenous to a point on $D$.
\end{thm: ex2}

More generally, we ask:

\begin{question*}
Do there exist curves $C,D \subset Y(1)_S^n$ for $S = \operatorname{Spec} \mathbb F_q$ (for $n \geq 3$) with projection to any factor dominant such that there are only finitely many points on $C$ isogenous to some point on $D$?
\end{question*}

Given the heuristic, one would expect the above question to have a negative answer for $n = 3,4$ and a positive answer for $n\geq 5$. It would be very interesting to prove either statement!

The above examples can be considered as examples of \emph{unlikely intersections}, i.e., the dimensions of the subvarieties add up to less than the dimension of the ambient space. Next, we consider the regime of \emph{just-likely intersections} and prove a positive result.

Consider two families of elliptic curves $\mathcal E, \mathcal F$ relative to $C$, a curve over a finite field or the ring of integers of a number field (possibly localized at some finite set of primes). In these cases, it is known under mild restrictions on the two families that there are infinitely many geometric points $x \in C$ such that the fibers $\mathcal E_x$ and  $\mathcal F_x$ are isogenous to each other (\cite[Proposition 7.3]{chai-oort},\cite{charles2018exceptional}).

To better understand this result, let us reformulate it in terms of moduli spaces. A family of elliptic curves over $C$ gives rise to a map 
\begin{align*}
    C &\to X(1)\\
    x &\to j(\mathcal E_x)
\end{align*}
where $X(1)$ is the proper modular curve so that our set up above corresponds to a map $f: C \to X(1)\times X(1)$. The geometric points $x \in C$ such that $\mathcal E_x$ is isogenous to $\mathcal F_x$ correspond exactly to the intersection of $f(C)$ with $T_N(\Delta)$, a Hecke translate of the diagonal $\Delta \subset X(1)^2$. Since $f(C)$ and $T_N(\Delta)$ have complimentary dimension, we expect there to be finitely many points in this intersection for any fixed $N$ and the question is about the behaviour as $N \to \infty$.

Such results have been generalized to families of abelian surfaces (\cite{maulik2022reductions},\cite{shankar2020exceptional}) and K$3$ surfaces (\cite{maulik2022picard},\cite{shankar2022exceptional}). Briefly, these papers consider the intersection of a curve $C$ with a "special divisor" $D$ in a Shimura variety and show that the number of distinct points in the intersection $C \cap T_N(D)$ grows without bound for $\{T_N: N \geq 1\}$ a sequence of Hecke translations. To do this, they estimate the local and global intersection numbers $(C \cap T_N(D))_x, C \cap T_N(D)$ as $N \to \infty$ and show that the second dominates the first (uniformly in $x$) by more than a constant factor. These two facts together prove the desired result.

Crucially, they prove the local estimate by appealing to the moduli interpretation of the associated Shimura variety and applying deformation theoretic techniques while the global intersection number computation is related to the Fourier coefficients of a modular form using the "specialness" of $D$.

Despite the proofs relying heavily on such moduli theoretic considerations, these problems can be naturally phrased for arbitrary subvarieties of a Shimura variety. Indeed, this was done by an AIM group working under Shou-Wu Zhang (\href{https://aimath.org/pastworkshops/intersectshimuraV.html}{Section 3.4 of the report}) and verbally communicated to the author by Ananth Shankar.

\begin{conjecture}\label{AIMconjecture}
Let $\mathcal S$ denote a simple Shimura variety (over some base), and suppose that $V,W$ are generically ordinary subvarieties having complementary dimension. Then the set of points in $V$ isogenous to some point of $W$ is Zariski-dense in $\mathcal S$. Further, the subset of $V\times W$ consisting of pairs of isogenous points is dense in $V\times W$.
\end{conjecture}

In this generality, the conjecture is known in very few cases. Perhaps most significantly, it is known in complete generality in the geometric case, i.e., over $\mathbb C$ by \cite{tayou2021equidistribution}. The aforementioned work \cite{charles2018exceptional} deals with the case where $\mathcal S = Y(1)$ relative to $\mathscr O_k$, the ring of integers of a number field (using the same strategy of proof as described before). In all other cases, it seems very hard to directly extend the older methods of proof to tackle the above conjecture when $V,W$ are allowed to be arbitrary subvarieties. 

Nevertheless, in this paper we prove the conjecture for $\mathcal S = Y(1)^n$ (parametrizing families of completely split $n$-dimensional abelian varieties). For $S$ the spectrum of a finite field, number field or the ring of integers in a number field (which covers most of the cases of interest), we consider the Shimura variety relative to $S$, denoted by $Y(1)^n_{S}$. In this case, two field valued points $x,y \in Y(1)^n$ are isogenous to each other if the corresponding abelian varieties $A_x,A_y$ parameterized by $x,y$ are geometrically isogenous to each other.

Since $Y(1)^n$ is not simple, some care has to be taken in formulating the conjecture. For $I \subset \{1,2,\dots,n\}$, let $\p_I: Y(1)^n \to Y(1)^I$ denote the projection map onto the corresponding co-ordinates. An obstruction to Conjecture \ref{AIMconjecture} comes from considering the same conjecture for $\p_I(V),\p_I(W) \subset Y(1)^I$. Another obstruction comes from the existence of supersingular elliptic curves over a finite field.

\theoremstyle{plain}
\newtheorem*{thm: main theorem intro}{Theorem \ref{cor: two varieties}}
\begin{thm: main theorem intro}
Let $V,W \subset Y(1)^n_S$ be (geometrically integral) varieties such that for every $I \subset \{1,\dots,n\}$ and $\p_I: Y(1)^n \to Y(1)^{|I|}$ the projection onto the factors indexed by $I$, we have
\[\dim(\p_I(V))+\dim(\p_I(W)) \geq \dim Y(1)^{|I|}\]
(where dimension means absolute dimension in all cases). Moreover if $S = \operatorname{Spec}\mathbb F_q$, we also suppose that $\p_i(V)$ has a dense set of points isogenous to some point on $\p_i(W)$ for all $i \leq n$. This is satisfied if one of  $\pi_i(V),\pi_i(W)$ is not equal to a supersingular point or alternatively, if both $\pi_i(V),\pi_i(W)$ are equal to a supersingular point.

Then the subset of points on $V\times W$ consisting of pairs of isogenous points is dense in $V\times W$.
\end{thm: main theorem intro}

Our proof of the above theorem will be inductive. The idea behind the proof is explained at the beginning of Section \ref{sec: proof main thm} after establishing the required notation. We have the following two corollaries of the above theorem in which we can drop the condition on the dimensions of the projections.

\newtheorem*{cor: dci}{Corollary \ref{cor: divisor curve intersection}}

\begin{cor: dci}
Suppose $D \subset Y(1)^n_S$ is a variety of absolute dimension $n-1$ (or a divisor) and $C \subset Y(1)^n_S$ is a curve relative to $S$ (or respectively of absolute dimension $1$) such that the projections $\p_i(D)\times_S\p_i(C) \subset Y(1)^2_S$ have dense intersection with Hecke translates of the diagonal. 

Then the subset of points on $C\times_S D$ consisting of pairs of isogenous points is dense in $C\times_S D$.
\end{cor: dci}

As a concrete example of this corollary, one could consider a non-isotrivial product of elliptic curves $\mathcal E_1\times\dots\times \mathcal E_n$ over a curve $C/\mathbb F_q$. Then, for any non-constant polynomial $F(t_1,\dots,t_n)$ in $n$ variables, the corollary shows that there are infinitely many points $P \in C(\overline{\mathbb F}_q)$ such that $F(j(\mathcal E_{1,P}),\dots,j(\mathcal E_{n,P})) = 0$.

\newtheorem*{cor: k isogeny classes}{Corollary \ref{rmk: k isogeny classes}}

\begin{cor: k isogeny classes}
Suppose $V \subset Y(1)^n_S$ is a variety of absolute dimension $d$ and let $Z \subset V$ be the set of points $v\in V$ such that $\{t_1(v),\dots,t_n(v)\}$ define at most $n-d$ distinct isogeny classes. If $S$ is the spectrum of a finite field, suppose further that $\p_i(V)$ is not a supersingular point. 

Then $Z$ is dense in $V$.
\end{cor: k isogeny classes}

\begin{notation}
In this paper, $S$ will denote the spectrum of a finite field, number field or the ring of integers of a number field. A variety $X \to S$ is defined to be a finite-type scheme such that the structure map is dominant when restricted to any irreducible component of $X$. The dimension of $X$ is always considered to be the \emph{absolute} dimension so that $\operatorname{Spec} \mathbb Z$ has dimension 1. We use $k$ to denote a finite field or number field, excluding the case of the ring of integers of a number field.

Since our results will be invariant under birational equivalence and geometric, i.e., invariant under replacing $S$ by an integral extension $S'$, one can assume that $X$ is smooth or projective and geometrically integral if convenient.
\end{notation}

\begin{notation}
We will also define $X_{n,S} = X(1)^n_S$, the n-fold product of the modular curve relative to $S$ and sometimes drop the subscript $S$ if it is clear from the context. For psychological reasons, we prefer to work with $X_n$ instead of $Y(1)^n$ - the difference is immaterial since our results are invariant under birational isomorphisms. Since our questions are geometric, there is also no loss of generality in assuming that $X(1)$ is the associated coarse moduli space - i.e., isomorphic to $\mathbb P^1$. \emph{We say that two points $x,y \in X_n$ are isogenous if the corresponding semi-abelian varieties $A_x,A_y$ parametrized by $x,y$ are geometrically isogenous.}
\end{notation}

\subsection{Acknowledgements} To be filled in after the referee process is completed.

\section{Unlikely intersections in high codimensions}\label{sec: low dim case}

In this section, we consider the following question.

\begin{question}
Suppose $V,W \subset X_n$ such that $\dim(V)+\dim(W) < n$. Is $\overline{\I(V,W)} \subset V$ Zariski-dense?
\end{question}

In the number field setting, if one of $V,W$ is special, we expect a negative answer to the above result by the Zilber-Pink conjecture\cite[Conjecture 2.2]{habegger2016minimality}. \textit{We therefore restrict henceforth to the case of $k = \mathbb F_q$ a finite field.} Surprisingly, we find families of examples providing a positive answer to the above question. These examples seem to have a very arithmetic flavour, distinct from the case where the dimensions are complementary.

\begin{theorem}\label{lem: postive low dim example}
Let $V = (t^{a_1},\dots,t^{a_n})$ and $W = (s^{b_1},\dots,s^{b_n})$ (for $a_i,b_i \geq 2$ and without loss of generality, coprime to $p$) be two rational curves in $X_{n,S}$ for $S = \operatorname{Spec} \mathbb F_q$. 

If the numbers $b_i/a_i$ are all powers of a fixed rational number $\gamma$, then $\I(V,W)$ is an infinite set.

Moreover, if we assume Artin's conjecture on primitive roots, then $\I(V,W)$ is infinite for general $a_i,b_i$.
\end{theorem}
\begin{proof}
We will restrict ourselves to isogenies of degree $p^n$ and also suppose that $t=s$. For $t \in \overline{\mathbb F}_p^\times$, $(t^{a_1},\dots,t^{a_n})$ is isogenous to $(t^{b_1},\dots,t^{b_n})$ through a $p$-power isogeny if there exist $N_1,\dots,N_n \in \mathbb N$ such that
\[t^{a_ip^{N_i}} = t^{b_i} \iff t^{a_ip^{N_i} - b_i} = 1.\]
Suppose $t$ has multiplicative order $\lambda$. Then the above is equivalent to requiring that
\[a_i/b_i \equiv p^{N_i} \pmod \lambda.\]
Said another way, we are interested in $\lambda$ such that $a_i/b_i$ lie in the subgroup of $(\mathbb Z/\lambda\mathbb Z)^\times$ generated by $p$. Recall Artin's conjecture which guarantees us that for any prime $p$, there exist infinitely many primes $\lambda$ such that $p$ is a primitive generator modulo $\lambda$. For all such $\lambda$, $a$ and $b$ are certainly powers of $p \pmod \lambda$ and moreover, we can find a $\lambda$ root of unity in $\overline{\mathbb F}_p$ as long as $\lambda$ is coprime to $p$.

In the case where $b_i/a_i=\gamma^{d_i}$ for some $\gamma$ (and all $i$), it is sufficient to require that $\gamma$ is in the subgroup generated by $p \pmod{\lambda}$ (and we no longer require $\lambda$ to be prime). We can obtain such $\lambda$ by considering the divisors of $p^m-\gamma$ as $m\to\infty$ and noting that since $|p^m-\gamma| \to \infty$, the number of distinct divisors also necessarily has to go to infinity.
\end{proof}

\begin{remark}
We note the striking fact that $V$ and $W$ are curves independent of the dimension $n$. That is, we have found two curves contained in a space of arbitrarily high dimension that have infinitely many common isogenous points.
\end{remark}

We give another positive example of a similar flavour. 

\begin{theorem}\label{thm: ex2}
Let $C = (t+b_0,t+b_1\dots,t+b_n), \Delta = (s,\dots,s) \subset X_{n+1,S}$ such that the $b_i\in \mathbb F_q$ are pairwise distinct and satisfy:
\[\frac{b_i-b_0}{b_j-b_0} \in \mathbb F_p\]
for all $1 \leq i,j \leq n$. Then there are infinitely many points on $C$ isogenous to a point on $\Delta$.
\end{theorem}

\begin{remark}
We note that the constraints on the $b_i$ linearly constrain the largest possible $n$ in terms of $p$.
\end{remark}

\begin{proof}
It is sufficient to find $t \in \overline{\mathbb F}_q$ and $u_1,\dots,u_n \in \mathbb N$ such that
\[t + b_0 = t^{q^{u_1}} + b_1 = \dots = t^{q^{u_n}} + b_n\]
or equivalently, for all $i\geq 1$,
\[t^{q^{u_i}} = t + b_i-b_0\]
Let $d \in \mathbb N_{>0},\lambda = b_1-b_0 $ and $s \in \overline{\mathbb F}_q$ be a solution to
\[s^{q^d} = s+\lambda.\]
By assumption, we have
\[\frac{b_i-b_0}{b_1-b_0} \in \mathbb F_p\]
and we define $u_i'$ to be any positive integer that reduces to the above fraction. Then if we let $u_i = du_i'$ we have
\[s^{q^{u_i}} = s + u_i'\lambda = s + b_i-b_0\]
and therefore, such a $s$ provides an example of a solution to the simultaneous equations in $t$ we started off with. 

Finally, we note that any $s$ such that $s^{q^d} = s+\lambda$ belongs to $\mathbb F_{q^{pd}} - \mathbb F_{q^d}$ since $s$ has an orbit of size $p$ under the Frobenius $x \to x^{q^d}$. Therefore, if we let $d = p,p^2,\dots$ range over the powers of $p$, the resulting $s$ will all be pairwise distinct thus proving that there exist infinitely many $t \in \overline{\mathbb F}_{q}$ that give rise to a point on $C$ isogenous to a point on $D$.

\end{proof}

\section{Just-likely intersections in complementary codimension}

Let $Z \subset X_{1,S}^2 \cong (\mathbb P^1)^2$ (with co-ordinates $t,s$) be a closed subscheme relative to $S$. Consider the set of points on $Z$ \emph{isogenous} to a point on the diagonal defined by
\[\I(Z,\Delta) = \{z \in Z: t(z) \text{ is isogenous to } s(z)\},\]
i.e., the subset of $z \in Z$ where the two elliptic curves parametrized by $z$ are isogenous. If we define
\[\Gamma_n = \{(t,s) \in X_1^2: t \text{ is isogenous to } s  \text{ by an isogeny with cyclic kernel of order } n\},\]
then
\[\I(Z,\Delta) = \bigcup_{n\geq 1}\Gamma_n\cap Z\]
is a (countable) union of closed subschemes of $Z$ since the $\Gamma_n$ are Hecke translates of the diagonal $\Delta$ and are hence closed.

\begin{definition}
We say that $Z$ is \textit{good} if $\mathscr I(Z,\Delta)$ is dense in $Z$ (and \textit{bad} otherwise).
\end{definition}

We can use a result of Chai-Oort \cite[Proposition 7.3]{chai-oort} to classify the good subschemes $Z$ over $S = \operatorname{Spec} \mathbb F_q$.

\begin{theorem}\label{thm: chai-oort}[Chai-Oort]
An irreducible $Z \subset X_{1,S}^2$ is good exactly in the following cases.
\begin{enumerate}
    \item $\dim(Z) = 0$. $Z = \{z\}$ and $t(z)$ is isogenous to $s(z)$.
    \item $\dim(Z) = 1$. $Z$ is not of the form $X_1\times\{p\}$ or $\{p\}\times X_1$ for $p$ corresponding to a supersingular elliptic curve.
    \item $\dim(Z) = 2$. $Z = X_1^2$. 
\end{enumerate}
\end{theorem}
\begin{proof}
The dimension $0$ case is trivial. In the dimension $2$ case with $Z = X_1^2$, we have $\Gamma_n \subset \I(Z,\Delta)$ and since the $\Gamma_n$ are all distinct divisors, their closure is all of $X_1^2$.

When $\dim(Z) = 1$, Proposition 7.3 of \cite{chai-oort} proves that $Z$ is good when it is not of the form $X_1\times\{p\}$ or $\{p\}\times X_1$ for $p$ any point. When $Z$ is of this form and $p$ corresponds to an ordinary elliptic curve, we see that $Z$ is good since there are infinitely many elliptic curves isogenous to an ordinary one.

Finally, if $Z$ is of this form and $p$ corresponds to a supersingular elliptic curve, $Z$ is not good since the only elliptic curves isogenous to a supersingular one are the other supersingular curves and there are only finitely many such curves.
\end{proof}

Similarly, over $S = \operatorname{Spec} k$ for $k$ a characteristic $0$ field (in particular a number field), one has:

\begin{theorem}\label{thm: Good in char 0 fn field}
An irreducible $Z\subset X_{1,S}^2$ is good except when $\dim Z = 0$ and $t(z)$ is not isogenous to $s(z)$.
\end{theorem}

The above statement follows from the analogous statement for the Noether-Lefschetz locus in a non trivial, one parameter variation of Hodge structures due independently to Green \cite[Proposition 17.20]{voisin2002theorie} and \cite{oguiso2003local}.

Finally, in the case of $S$ the spectrum of the ring of integers of a number field $k$, Francois Charles \cite{charles2018exceptional} proves the following theorem.

\begin{theorem}\label{thm: Charles}
For any two (generalized) elliptic curves $\mathcal E_1,\mathcal E_2$ relative to $S$, the spectrum of a ring of integers of a number field $k$, there exist infinitely many primes $\mathfrak p \subset k$ such that $\mathcal E_1\pmod{\mathfrak p}$ is isogenous to $\mathcal E_2\pmod{\mathfrak p}$.
\end{theorem}

In this paper, we consider a higher dimensional version of the above questions. For $V,W \subset X_{n,S}$, we define
\[\I(V, W) = \{v \in V: \text{there exists } w\in W \text{ such that } v \text{ is isogenous to } w \} \subset V.\]
\begin{question}\label{ques: pair of varieties}
Suppose $\dim(V) + \dim(W) \geq \dim(X_n)$. Is $\I(V,W)$ dense in $V$?
\end{question}
As stated, the question might have a negative answer because the projections of $V,W$ to a proper subset of the co-ordinates might have very low dimension. The question might also have a negative answer because the projections of $V,W$ to any of the $X_1$ factors might fail to have dense intersection (up to Hecke translation). We show that these are the only two obstructions. More precisely, let us consider the following setup.

Consider $X_{n,S}^2$ with co-ordinates $t_1,\dots,t_n,s_1,\dots,s_n$ and for any $I = \{i_1,\dots,i_r\} \subset \{1,2,\dots,n\}$, define the projection
\[\pi_I : X_{n,S}^2 \to X_{I,S}^2, P \to (t_{i_1}(P),\dots,t_{i_r}(P),s_{i_1}(P),\dots,s_{i_r}(P)).\]
For $J = \{1,2,\dots,n\}-I$, we set $\wpi_I = \pi_J$. Let $Y \subset X_{n,S}^2$ be a subvariety with the following properties:
\begin{enumerate}\label{list: suff properties}
    \item \underline{Y has "good projections".} The image $\pi_i(Y)$ is good for all $i=1,\dots,n$.
    \item \underline{Y has "enough dimension".} The image $\pi_I(Y)$ has (absolute) dimension at least $|I|$ for all subsets $I$ of $\{1,\dots,n\}$ with $|I| \geq 2$.
\end{enumerate}

Then if we let
\[\Delta = V(t_1=s_1,\dots,t_n=s_n) \cong X_n \subset X_n^2  \]
and
\[\I(Y,\Delta) = \{x \in Y : t_i(x) \text{ is isogenous to } s_i(x) \text{ for all } i\}\subset Y\]
we have
\begin{theorem}\label{thm: main, diag version}
For $Y$ satisfying the above properties, $\I(Y,\Delta)$ is dense in $Y$. We note that the condition on having "good projections" is automatic for $S$ the spectrum of the ring of integers in a number field.
\end{theorem}

\begin{remark}
As discussed above, it is necessary that $Y$ have "good projections" so that $\I(Y,\Delta)$ is not contained in a proper closed subscheme of $Y$. Similarly, it's also necessary that $\I(\pi_I(Y),\Delta)$ be dense in $\pi_I(Y)$. While we don't quite require this, we require the stronger (and easier to check condition) of $Y$ having "enough dimension" so that we can argue inductively.
\end{remark}

\begin{notation}
We note that for $I \subset \{1,\dots,n\}$, we use $\pi_I: X_n^2 \to X_{I}^2$ and $p_I: X_n \to X_I$. Similarly, when we have two families of products of elliptic curves in mind, we use $X_n^2$ and if we have one family in mind, we use $X_n$.
\end{notation}

We prove Theorem \ref{thm: main, diag version} in Section \ref{sec: proof main thm}. For the rest of this section, we assume this theorem and derive some consequences.

\begin{theorem}\label{cor: two varieties}
Let $V,W \subset X_{n,S}$ be varieties such that for every $I \subset \{1,\dots,n\}$ and $\p_I: X_{n,S} \to X_{|I|,S}$ the co-ordinate projection onto the factors indexed by $I$, we have
\[\dim(\p_I(V))+\dim(\p_I(W)) \geq \dim X_{|I|}\]
(where dimension means absolute dimension in all cases). Moreover if $S = \operatorname{Spec}\mathbb F_q$, we also suppose that $\p_i(V)$ has a dense set of points isogenous to some point on $\p_i(W)$ for all $i \leq n$.

Then the subset of points on $V\times W$ consisting of pairs of isogenous points is dense in $V\times W$.
\end{theorem}
\begin{proof}
Take $Y = V\times_S W \subset X_n\times_S X_n$ and apply Theorem \ref{thm: main, diag version}. To check that the dimension constraint is satisfied, note that $\dim Y = \dim V + \dim W -\dim S \geq \dim X_n - \dim S = n$ (and similarly for the co-ordinate projections). The condition on good projections is easily seen to hold too.

\end{proof}

\begin{corollary}\label{cor: divisor curve intersection}
Suppose $D \subset X_{n,S}$ is a variety of absolute dimension $n-1$ (or a divisor) and $C \subset X_{n,S}$ is a relative curve over $S$ (or of absolute dimension $1$ respectively) such that the projections $\p_i(C)\times_S\p_i(D) \subset X(1)^2_S$ have dense intersection with Hecke translates of the diagonal.

Then the subset of points on $C\times_S D$ consisting of pairs of isogenous points is dense in $C\times_S D$.
\end{corollary}

\begin{proof}

The point is that we can reduce to the case where the dimension of $C$ does not drop under any projection. More precisely:

For any $I \subset \{1,\dots,n\}$, $\p_I(D)$ has absolute dimension at least $|I|-1$ (or contains a divisor). Moreover, we can easily reduce to the case where $\p_i(C)$ is not a point for any $i\leq n$. Once we make this reduction, $\p_I(C)$ is a relative curve for all $I$ (or has absolute dimension $1$) and therefore
\[\dim(\p_I(D)) + \dim(\p_I(C)) = \dim X_{|I|},\]
so that $C\times_S D$ has "enough dimension". Note also that $C\times_S D$ has "good projections" by assumption so that we can apply Theorem \ref{thm: main, diag version}.

\end{proof}

\begin{corollary}\label{rmk: k isogeny classes}
Suppose $V \subset X_{n,S}$ is a variety of absolute dimension $d$ and let $Z \subset V$ be the set of points $v\in V$ such that $\{t_1(v),\dots,t_n(v)\}$ define at most $n-d$ distinct isogeny classes. If $S$ is the spectrum of a finite field, suppose further that $\p_i(V)$ is not a supersingular point. 

Then $Z$ is dense in $V$. 
\end{corollary}
\begin{proof}
Let 
\[W_0 = \{(t_1,\dots,t_{n-d},s_1,\dots,s_d) : t_i,s_j \in \mathbb P^1 \text{ and for each } j, s_j = t_i \text{ for some } i.\} \]
and $W$ be the $S_n$ orbit of $W_0$ under the standard action permuting the co-ordinates. By definition, we have $Z = \I(V,W)$. In this case, note that for any $I \subset \{1,\dots,n\}$,
\[\dim p_I(W) = \min(\dim X_{|I|},\dim X_{n-d})\]
while $\dim p_I(V) \geq \max\{0,d - (n-|I|)\}$. In any case, we see that $\dim p_I(V) + \dim p_I(W) \geq \dim X_{|I|}$ as required. Moreover, since $p_i(V)$ is not a supersingular point and $p_i(W) = X_1$, we see that $V\times W$ satisfies the condition of "good projections". 

Altogether, we can now apply Theorem \ref{cor: two varieties} to conclude the proof.
\end{proof}

\subsection{Proof of Zariski dense intersections}\label{sec: proof main thm}

In this section, we prove Theorem \ref{thm: main, diag version} by an induction on the ambient dimension (\textit{and restrict to $S = \operatorname{Spec} k$, the spectrum of a finite field or number field until Theorem \ref{thm: ring of integers of a number field case}}). We will use the notation defined just before Theorem \ref{thm: main, diag version} throughout. The base case of $n=1$ is of course Theorem \ref{thm: chai-oort} (Chai-Oort) for $k$ a finite field and Theorem \ref{thm: Good in char 0 fn field} for $k$ a number field. Our strategy will be to define infinitely many divisors
\[\{D_1,D_2,\dots\} \subset Y\]
so that there are infinitely many $r \in \mathbb N$ with $\pi_1(D_r)$ good and moreover, so that the $\wpi_1(D_r)$ satisfy the inductive hypothesis relative to $X_{n-1}$. This will prove that $\I(D_r,\Delta)$ is dense and thus, $\overline{\I(Y,\Delta)}$ will contain infinitely many distinct divisors $D_r$ and hence be equal to $Y$. We can suppose that $Y$ is geometrically integral by passing to an extension if needed and considering each component separately. We make a further reduction to supposing that $\dim(\pi_i(Y)) \geq 1$ for all $i=1,\dots,n$:

Suppose otherwise that $\dim \pi_i(Y) = 0$ and hence, $\pi_i(Y)$ is a good point. This is the simplest case since $t_i(y)$ is isogenous to $s_i(y)$ for all $y$. We can simply replace $Y$ by $\wpi_i(Y) \cong Y$ and conclude by the inductive hypothesis.

To define $D_r$, we consider $\pi_1(Y) \subset X_1^2$. By the assumption that $Y$ has "good projections", $\pi_1(Y)$ is good. We have the following possibilities:
\begin{enumerate}
    \item Suppose $\dim \pi_1(Y) = 1$. Since this is a good curve, we can find infinitely many points $E_1,E_2,\dots$ in $C$ such that $t_1(E_r)$ is isogenous to $s_1(E_r)$ (by an isogeny defined by $\Gamma_{n_r}$, say). Since $\dim \pi_1(Y) = 1$ in this case, the $E_r$ are also divisors and we define
    \[D_r = \pi_1^{-1}(E_r)\cap Y \subset Y.\]
    \item Suppose $\pi_1(Y) = X_1^2.$ In this case, we take $E_r = \Gamma_r$ (so $n_r = r$) and
    \[D_r = \pi_1^{-1}(E_r)\cap Y \subset Y.\]
\end{enumerate}

In both the cases above, it is clear that the $D_r$ can be chosen to be pairwise disjoint (by restricting to a component and throwing away some $r$ if needed) since their projections $E_r$ can be so chosen. We will also make the further simplifying assumption that the $D_r, E_r$ are irreducible (as can always be done by restricting to a component).

We have chosen $E_r$ in each case so that $t_1(e)$ is isogenous to $s_1(e)$ for all $e \in E_r$ ( and hence also for all $d\in D_r$). It remains to show that $F_r = \wpi_1(D_r)$ satisfies the inductive hypothesis. We prove that all but finitely many of the $F_r$ have "good projections" by using that $Y$ has "good projections".

\begin{lemma}\label{lem: F_r has good projections}
For all but finitely many $k$ and any $i\geq 2$, $\pi_i(F_r) = \pi_i(D_r) \subset X_1^2$ is good.
\end{lemma}
\begin{proof}

We will consider two cases depending on $\dim \pi_i(Y) \in \{2,1\}$. 

First suppose that $\dim \pi_i(Y) = 2$ so that $\pi_i(Y) = X_1^2$. Note that there are only finitely many curves in $X_1^2$ that are not good, corresponding to $X_1\times\{p\}$ or $\{p\}\times X_1$ with $p$ supersingular in the positive characteristic case (and no bad curves in characteristic $0$). There will be at most finitely many divisors that get contracted to a point and finitely many other divisors that get contracted to the bad curves described above. We can simply exclude the finitely many $F_r$ that are in this set since the image of any other $F_r$ will be a good curve.

Suppose now that $\pi_i(Y)$ is a curve and let $R$ be the set of $r\in \mathbb N$ such that $\pi_i(F_r)$ is a bad point. We would like to show that this is a finite set. \textit{Assume for contradiction that $K$ is infinite.}

Under this assumption, we first prove that $\pi_1(Y)$ is also a curve: Since $\pi_i(Y)$ is a good curve, it contains infinitely many good points $p_1,p_2,\dots$. Let $Z_\ell = \pi_1^{-1}(p_\ell) \subset Y$. The $Z_\ell$ cannot intersect any of the $D_r$ for $k \in K$. Therefore, $\pi_1(Z_\ell)$ is a closed subscheme of $\pi_1(Y)$ that avoids the divisor $E_r = \Gamma_{n_r}$. In the Picard group $\Gamma_r$ has class $(\phi(r),\phi(r))$\footnote{The degree of the $r$-th Hecke translate is denoted $\phi(r)$ and all we need is that it is a positive number.} and is an ample divisor, therefore it intersects every effective divisor non trivially. This forces $\pi_1(Z_\ell)$ to be a point for all the $\ell$ and since $\pi_1$ contracts infinitely many divisors (the $Z_\ell$) to a point, the image $\pi_1(Y)$ has to be a curve. 

In particular, since $\pi_1(Y)$ is a curve, $E_r \in X_1^2$ is a point for all $r$. Now, consider the map $\pi_{1i}: Y \to \pi_1(Y)\times\pi_i(Y)$. The divisors $D_r$ for $r \in R$ get contracted to the point $E_r$ along $\pi_1$ and by assumption, the $D_r$ also get contraced to a (bad) point along $\pi_i$. Therefore, $\pi_{1i}$ contracts $D_r$ for all $r\in R$ and since this is an infinite set, $\pi_{12}(Y)$ is forced to be a curve and has dimension $1 < 2$, contradicting our assumption that $Y$ has "enough dimension" with $I = \{1,i\}$. 

\end{proof}

Next, we show that the $F_r$ have "enough dimension". We first prove that $\dim(\pi_I(D_r))$ is large when $I = \{2,3,\dots,n\}$.

\begin{lemma}\label{lem: Dk dim drop by 1}
For $I = \{2,3,\dots,n\}$ and all but finitely many $r$:
\[\dim(\pi_I(D_r)) \geq n-1.\]
\end{lemma}
\begin{proof}
The proof breaks up into three cases depending on the dimension of the generic fiber of $\pi_I$ restricted to $Y$.

Suppose first that $\pi_I$ restricted to $Y$ has generic fibers of dimension $0$. In this case, there will be some proper closed locus on $Y$ such that the fiber dimension of $Y$ is at least $1$. We exclude the finitely many $D_r$ contained in this closed locus. For any other $D_r$, the generic fiber dimension along $\pi_I$ will be $0$ and hence
\[\dim(\pi_I(D_r)) = \dim(D_r) = \dim(Y)-1 \geq n-1.\]

Next, suppose that $\pi_I$ restricted to $Y$ has generic fibers of dimension $1$. Once again, there will be some proper closed locus on $\pi_I(Y)$ for which the fiber dimension is at least $2$. We exclude the finitely many $D_r$ mapping into this closed locus. For any other $Q \in \pi_I(Y)$, the fiber will be a curve $C_Q$ that maps injectively to $X_1^2$ along $\pi_1$. Then $\Gamma_{n_r} \subset X_1^2$, being an ample divisor in $\pi_1(Y)$, intersects $\pi_1(C_Q)$ non trivially. In other words, $Q \in \pi_I(D_r)$ for a generic point $Q$ and hence
\[\dim(\pi_I(D_r)) = \dim(\pi_I(Y)) = \dim(Y)-1 \geq n-1.\]

Finally, suppose that the generic fiber dimension of $\pi_I$ restricted to $Y$ is $2$. In this case, $Y = X_1^2 \times \pi_I(Y)$ and $D_r = \Gamma_{n_r}\times \pi_I(Y)$ so that $\pi_I(D_r) = \pi_I(Y)$ and 
\[\dim(\pi_I(D_r)) = \dim(\pi_I(Y)) \geq n-1\]
by the assumption of "enough dimension" for $Y$.

\end{proof}

Now, we apply the lemma to prove that the $F_r$ have "enough dimension" along arbitrary projections.

\begin{lemma}\label{lem: F_r have enough dimension}
For all but finitely many $F_r$ and $I \subset \{2,\dots,n\}$ with $|I| \geq 2$, we have
\[\dim(\pi_I(F_r)) \geq |I|.\]
\end{lemma}
\begin{proof}
Let $J = \{1\}\cup I$ and let us decompose the projection $\pi_I$ as $\pi_I\circ\pi_J$. By the assumption that $Y$ has "enough dimension" along projections, $\dim \pi_J(Y) \geq |J| = |I|+1$. Moreover, $\pi_J(D_r)$ is a divisor on $\pi_J(Y)$ since we have $\pi_J(D_r) = \pi_{1,J}^{-1}(E_r)$ for
\begin{align*}
    \pi_{1,J}:\pi_J(Y) &\to X_1^2\\
                    y &\to (t_1(y),s_1(y)).
\end{align*}
In particular, we can identify $\pi_I(D_r)$ with $\pi_I(\pi_J(D_r))$ and applying the previous lemma with $\pi_J(Y),\pi_J(D_r)$ playing the role of $Y,D_r$ respectively, we see that
\[\dim(\pi_I(F_r)) = \dim(\pi_I(D_r)) = \dim(\pi_J(Y))-1 \geq |J|-1 = |I|.\]

\end{proof}

Putting everything together

\begin{proof}[Proof of Theorem \ref{thm: main, diag version} when $S = \operatorname{Spec} k$ for $k$ a finite field or number field.]
Let us consider the sequence of divisors $D_1,D_2,\dots$ constructed above. By Lemmas \ref{lem: F_r has good projections} and $\ref{lem: F_r have enough dimension}$, we have proven that the $F_r = \wpi_1(D_r)$ satisfy the inductive hypothesis for infinitely many $r$ and for these $r$, we thus have $F_r \in \overline{\I(\wpi_1(Y),\Delta)}$. 

Moreover, we have $s_1(x) = t_1(x)$ for $x \in D_r$ by fiat. Therefore, $D_r \in \overline{\I(Y,\Delta)}$ for infinitely many $r$. Since $\overline{\I(Y,\Delta)}$ contains infinitely many distinct divisors, this proves that it is equal to $Y$, completing the entire proof.
\end{proof}

Finally, we prove Theorem \ref{thm: main, diag version} in the case of $\mathscr O_k$ the ring of integers of a number field $k$ using Theorem \ref{thm: Charles}. In fact, this case is easier than the previous case considered where $S$ is the spectrum of a field since having "good projections" is automatic over $\mathscr O_k$ by Theorem \ref{thm: Charles}. 

\begin{theorem}\label{thm: ring of integers of a number field case}
Suppose we have a variety $Y \subset X_{n,S}\times_{S} X_{n,S}$ with $S = \operatorname{Spec} \mathscr O_k$ the ring of integers of a number field such that for all $I \subset \{1,\dots,n\}$,
\[\dim(\pi_I(Y)) \geq |I|\]
(where dimension means absolute dimension as usual).

Then $\I(Y,\Delta)$ is dense in $Y$.
\end{theorem}
\begin{proof}
We will prove the theorem by induction. Since the inductive hypothesis is valid for $\wpi_1(Y) \subset X_{n-1}\times_{\mathscr O_k}X_{n-1}$, we conclude that $\I(\wpi_1(Y),\Delta)$ is Zariski dense in $\wpi_1(Y)$. Moreover for any $w \in \wpi_1(Y)(\mathscr O_k)$, the fiber $\wpi_1^{-1}(w)$ considered as a subset of $X_1^2$ under the $\pi_1$ projection is either a $\mathscr O_L$ valued point (for $k \subset L$ a finite extension), a relative curve or all of $X_1^2$. 

In the first case, it is automatically good by Theorem \ref{thm: Charles} and in the other two cases, the generic point $\pi_1(\wpi_1^{-1}(w)_k)$ is good by Theorem \ref{thm: Good in char 0 fn field} and spreading out a generic isogeny shows that $\pi_1(\wpi_1^{-1}(w))$ is itself good. In any case, we have shown that if $w \in \I(\wpi_1(Y),\Delta)$, then $\overline{\wpi_1^{-1}(w)} = \overline{\I(\wpi_1^{-1}(w),\Delta)}$. Together with the fact that $\I(\wpi_1(Y),\Delta)$ is dense in $\wpi_1(Y)$ by the inductive hypothesis, this completes the proof.
\end{proof}

\bibliographystyle{alpha}

\bibliography{sample.bib} 


\end{document}